\documentclass[10pt]{article}
\usepackage{amssymb,amsthm}
\usepackage[usenames]{color}
\usepackage[utf8]{inputenc}
\usepackage[english]{babel}
\usepackage{amsmath,amsfonts,amscd,amssymb,amsthm,latexsym,eucal,graphics,amsbsy} 
\usepackage{enumerate}
\usepackage{hyperref}
\newtheorem{lem}{Lemma}[section]
\newtheorem{theorem}{Theorem}[section]

\newtheorem{cor}{Corollary}[section]
\newtheorem*{rmk}{Remark}
\newtheorem{prop}{Proposition}[section]
\newtheorem{Problem}{Problem}[section]

\newcommand{\Sym}{\operatorname{Sym}}
\newcommand{\Alt}{\operatorname{Alt}}
\newcommand{\Sp}{\operatorname{Sp}}
\newcommand{\SU}{\operatorname{SU}}
\newcommand{\Ort}{\operatorname{O}}
\newcommand{\Fi}{\operatorname{Fi}}
\newcommand{\Pom}{\operatorname{P\Omega}}
\newcommand{\wre}{\operatorname{wr}}

\usepackage{graphicx} 
\usepackage[left=2.5cm,right=2.5cm,top=2.4cm,bottom=2.4cm,bindingoffset=0cm]{geometry}

\title{On Jordan algebras that are factors of Matsuo algebras}
\author{Ilya Gorshkov, Andrey Mamontov and Alexey Staroletov\footnote{The work is supported by Mathematical Center in Akademgorodok under agreement No.075-15-2022-281 with the Ministry of Science and Higher Education of the Russian
Federation}}

\date{\vspace{-30px}}

\begin{document}

\maketitle
\newcommand{\Addresses}{{
		\bigskip\noindent
		\footnotesize
	    Ilya~Gorshkov, \textsc{Sobolev Institute of Mathematics, Novosibirsk, Russia;}\\\nopagebreak
        \textsc{Novosibirsk State Technical University, Novosibirsk, Russia;}\\\nopagebreak
		\textit{E-mail address: }
        \texttt{ilygor8@gmail.com}

        \medskip\noindent
        Andrey~Mamontov, \textsc{Sobolev Institute of Mathematics, Novosibirsk, Russia;}\\\nopagebreak
		\textsc{Novosibirsk State University, Novosibirsk, Russia;}\\\nopagebreak
		\textit{E-mail address: } \texttt{andreysmamontov@gmail.com}

		\medskip\noindent
		Alexey~Staroletov, \textsc{Sobolev Institute of Mathematics, Novosibirsk, Russia;}\\\nopagebreak
        \textsc{Novosibirsk State University, Novosibirsk, Russia;}\\\nopagebreak
		\textit{E-mail address: } \texttt{staroletov@math.nsc.ru}

		\medskip
}}

\begin{abstract}
We describe all finite connected 3-transposition groups whose Matsuo algebras have nontrivial factors that are Jordan algebras. As a corollary, we show that if $\mathbb{F}$ is a field of characteristic 0, then there exist
infinitely many primitive axial algebras of Jordan type $\frac{1}{2}$ over $\mathbb{F}$ that are not factors of Matsuo algebras. As an illustrative example, we prove this for an exceptional Jordan algebra over $\mathbb{F}$.
\end{abstract}



\section{Introduction}

Axial algebras of Jordan type were introduced by Hall, Rehren, and Shpectorov~\cite{hrs}
within the framework of the general theory of axial algebras.
The main inspiration for this theory are the Griess algebra~\cite{Gr82}, Majorana theory~\cite{I09}, and algebras associated with $3$-transposition groups~\cite{Matsuo}.
Modern results and open problems in the theory of axial algebras can be found in a recent survey~\cite{survey}.

Consider a commutative $\mathbb{F}$-algebra $A$, where $\mathbb{F}$ is a field of characteristic not equal to two. For the element $a$ of $A$ and $\lambda\in\mathbb{F}$, the $\lambda$-eigenspace for the adjoint operator $ad_a$ on $A$ is denoted by $A_\lambda(a)$. An idempotent whose adjoint operator is semisimple will be called an {\it axis}. If $A$ is generated by a set of axes, then $A$ is an {\it axial algebra}.
An axis $a$ is {\it primitive} if $A_1(a)$ is one-dimensional, i.e. spanned by $a$.
Suppose that $\eta\in\mathbb{F}$ and $0\neq\eta\neq1$. The commutative $\mathbb{F}$-algebra $A$ is a {\it primitive axial algebra of Jordan type $\eta$} provided it is generated by a set of primitive axes with each member $a$ satisfying the following properties:
$$A=A_1(a)\oplus A_0(a)\oplus A_\eta(a),A_0(a)^2\subseteq A_0(a),$$
and for all $\delta,\epsilon\in\{\pm\}$,
$$A_\delta(a)A\epsilon(a)\subseteq A_{\delta\epsilon}(a), \text{ where }
A_+(a)=A_1(a)\oplus A_0(a)\text{ and }A_-(a)=A_\eta(a).$$
These properties generalize the Peirce decomposition for idempotents in Jordan algebras, where
$\frac{1}{2}$ is replaced with $\eta$. In particular, this explains the motivation for the name of this class of axial algebras.

Another basic example of axial algebras of Jordan type are Matsuo algebras.
They were introduced by Matsuo~\cite{Matsuo} and later generalized in~\cite{hrs}.
Recall that a group $G$ is a  {\it $3$-transposition group} if it is generated by a normal set $D$ of involutions such that the order of the product of any pair of these involutions is not greater than three. Let $\eta$, as before, be an element of $\mathbb{F}$ distinct from 0 and 1. The Matsuo algebra $M_\eta(G,D)$ has $D$ as its basis, where each element of $D$ is an idempotent. Moreover, the product in $M_\eta(G,D)$ of two distinct elements $c,d\in D$ equals 0 if $|cd|=2$ and  $\frac{\eta}{2}(c+d-c^d)$ if $|cd|=3$.
It turns out that $M_\eta(G,D)$ is a primitive axial algebra of Jordan type $\eta$ with generating set of primitive axes $D$~\cite{hrs}. Moreover, it is known that if $\eta\neq\frac{1}{2}$,
then every primitive axial algebra of Jordan type $\eta\neq\frac{1}{2}$ is a factor algebra of a Matsuo algebra \cite{hrs,hss}.

It was conjectured in \cite{GS} that any primitive axial algebras of Jordan type $\frac{1}{2}$
is a Jordan algebra or a factor of a Matsuo algebra.
De Medts and Rehren classified Matsuo algebras that are Jordan algebras \cite{TR2017}. As a consequence, it can be concluded that most Matsuo algebras are not Jordan.
The motivation for this paper is the following question: are there examples of axial algebras of Jordan type $\frac{1}{2}$ that are not factors of Matsuo algebras? We provide examples of such algebras among Jordan algebras.
We focus on Matsuo algebras corresponding to connected $3$-transposition groups $(G,D)$, i.e. where $D$ is a conjugacy class of $3$-transpositions. If $D$ is a union of conjugacy classes, then the Matsuo algebra on $D$ is the direct sum of the corresponding Matsuo algebras constructed from each conjugacy class contained in $D$~\cite{hrs}. We say that two nontrivial connected 3-transposition groups $(G_1,D_1)$ and $(G_2, D_2)$ have the same central type if
$G_1/Z(G_1)$ and $G_2/Z(G_2)$ are isomorphic as 3-transposition
groups. It is easy to see that if two 3-transposition groups have the same central type, then their Matsuo algebras are isomorphic.

It turns out that every Matsuo algebra $M=M_{\eta}(G,D)$, where $(G,D)$ is a connected 3-transposition group, has a maximal ideal $M^\perp$ containing every proper ideal of $M$. In fact, this ideal is the radical of a symmetric bilinear form on $M$ (see Section~\ref{sec:matsuo}).  Clearly, if an algebra is Jordan then every homomorphic image is Jordan. This implies that $M$ has Jordan factors if and only if $M/M^\perp$ is Jordan. In this paper we describe all algebras $M$ satisfying the latter condition.
\begin{theorem}\label{t:main}
Let $\mathbb{F}$ be a field of characteristic $0$ and $\eta\in\mathbb{F}\setminus\{0,1\}$.
Suppose that $(G,D)$ is a finite connected $3$-transposition group and $M=M_{\eta}(G,D)$ is the Matsuo algebra constructed by $(G,D)$ and $\eta$. If $J=M/M^\perp$ is a Jordan algebra, then
one of the following statements holds.
\begin{enumerate}
 \item[$(i)$] $G$ is the cyclic group of order $2$ and so $J=M$ is one-dimensional;
 \item[$(ii)$] the product of every two distinct elements of $D$ has order $3$, $\eta=2$, and $J$ is one-dimensional;
 \item[$(iii)$] $\eta=\frac{1}{2}$ and
 $G$ has the same central type as one of the following $3$-transposition groups: $\Sym(m)~(m\geq2)$, $2^{\bullet1}:\Sym(m)~ (m\geq4)$, $3^{\bullet1}:\Sym(m)~(m\geq4)$, $3^2:2$,
 $O_8^+(2)$, $O_6^-(2)$, $Sp_6(2)$, ${}^+\Omega_6^-(3)$, $SU_4(2)$, $SU_5(2)$, or $4^{\bullet1}SU_3(2)'$. In particular, $\dim J\in\{1, m^2, \frac{m(m-1)}{2}\}$, where $m\geq3$.
\end{enumerate}
Moreover, each of the possibilities in items $(i)-(iii)$ is realized for some $M$.
\end{theorem}

In Section~\ref{sec:proof}, we discuss possible Matsuo algebras $M$ satisfying the hypothesis of this theorem and the corresponding 3-transposition groups $(G,D)$ in detail (see Proposition~\ref{p:main}). Note that the case $M^\perp=0$ was considered in \cite{TR2017}. Moreover, it was mentioned in \cite[Remark~3.5]{TR2017}
that the Weyl groups for simply-laced root systems of types $E_n$ with $6\leq n\leq 8$ and $D_n$, considered as 3-transposition groups, correspond to Matsuo algebras that have among their factors the Jordan algebra (of dimension $\frac{n(n+1)}{2}$) of all symmetric $n\times n$ matrices.

Note that for every integer $n\geq1$ there exists a simple Jordan algebra of dimension $n$ which is a primitive axial algebra of Jordan type $\frac{1}{2}$. As an example one can take a so-called Jordan spin factor algebra (see, for example, \cite[Lemma~5.1]{hss}). This together with Theorem~\ref{t:main} implies the following corollary.

\begin{cor}
If $\mathbb{F}$ is a field of characteristic 0, then there exist infinitely many primitive axial algebras of Jordan type $\frac{1}{2}$ over $\mathbb{F}$ that are not factors of Matsuo algebras.
\end{cor}

In Section~\ref{sec:albert}, we present another example, which is the cornerstone in the theory of Jordan algebras. Given a field $\mathbb{F}$ of characteristic $0$,
we show that a $27$-dimensional Albert algebra over $\mathbb{F}$ is an axial algebra of Jordan type $\frac{1}{2}$ generated by four primitive axes. Theorem~\ref{t:main} implies that this
algebra, known to be a simple Jordan algebra, is not a factor of a Matsuo algebra.

Finally, we mention two results on the status of the conjecture formulated above.
Gorshkov and Staroletov proved that every axial algebra of Jordan type $\frac{1}{2}$
generated by at most three primitive axes is Jordan and has dimension not exceeding nine~\cite{GS}.
It has recently been proved that the dimension of a 4-generated algebra does not exceed 81, which is the dimension of a 4-generated Matsuo algebra~\cite{TRS}.
In the general case, it is not even known whether an axial algebra of Jordan type generated by a finite number of primitive axes has a finite dimension or not.

The proof of Theorem~\ref{t:main} is based on the classification of 3-transposition groups and the dimensions of the eigenspaces of diagrams on the corresponding sets of 3-transpositions. The necessary definitions and results are given in Section~\ref{sec:3-transp}. In Section~\ref{sec:matsuo}, we provide the necessary information on Jordan and Matsuo algebras. In Section~\ref{sec:wreath}, we give a convenient description of the $3$-transposition groups that are obtained from the symmetric group by the wreath product construction, these groups are a special case in the proof of Theorem~\ref{t:main}. Section~\ref{sec:proof} is devoted to this proof. Finally, in Section~\ref{sec:albert} we show that an Albert algebra over a field of characteristic different from $2$ and $3$ is a primitive axial algebra of Jordan type $\frac{1}{2}$.

\section{Preliminaries: 3-transposition groups}\label{sec:3-transp}

Suppose that $G$ is a group and $D$ is a normal set of involutions, i.e. a union of conjugacy classes of elements of order 2. If for every pair $d,e\in D$ the order of $de$ is at most 3, then $D$ is called a set of 3-transpositions. This notion was introduced by Fischer as a generalization of properties of transpositions in symmetric groups~\cite{Fi66}.

We say that $(G,D)$ is a 3-transposition group if $D$ generates $G$ and is a set of 3-transpositions.
If $S$ is a subset of $D$, the diagram of $S$, denoted $(S)$,
is the graph whose vertices are elements of $S$ with the pair $\{d,e\}$ forming an edge precisely when $|de|=3$. This notion is important for 3-transpositions groups since the subgroup of $G$ generated by $S$ is a homomorphic image of the Coxeter group with diagram $(S)$.

\begin{lem}{\em\cite[(1.2) and Lemma~(2.1.1)]{F}}\label{l:inherit} Suppose that $D$ is a set of $3$-transpositions in $G$. Then the following statements hold.
\begin{enumerate}[(i)]
\item If $H$ is a subgroup of $G$, then $D\cap H=\varnothing$ or $D\cap H$ is a set of $3$-transpositions in $H$. If $N$ is a normal subgroups of $G$, then $D\subset N$ or the nontrivial elements of $DN/N$ form a normal a set of $3$-transpositions in $G/N$.
\item Let $D_i$, for $i\in I$, be the connected components of $(D)$. Then each $D_i$ is a conjugacy class of $3$-transpositions in the group
 $G_i=\langle D_i\rangle$. Furthermore, the normal subgroup $\langle D\rangle$ is the central product of its subgroups $G_i$.
 \item If $G=\langle D\rangle$ then, for $d\in D\setminus Z(G)$ each coset $dZ(G)$ meets $D$ only in $d$.
\end{enumerate}

\end{lem}

It follows that 3-transposition groups $(G,D)$ with connected diagram $(D)$ are of prime interests.
We say that $(G,D)$ is a connected 3-transposition group if $(D)$ is connected.
Note that this is equivalent to $D$ being a conjugacy class of $G$.
We say that the two connected 3-transposition groups $(G_1,D_1)$ and $(G_2,D_2)$ have the same {\it central type} provided $G_1/Z(G_1)$ and $G_2/Z(G_2)$ are isomorphic as $3$-transposition groups. By Lemma~\ref{l:inherit}(iii), two connected 3-transpositions groups have the same central type if only if their diagrams are isomorphic.

Finite connected 3-transposition groups $(G,D)$ such that $O_2(G)O_3(G)\leq Z(G)$ were classified by Fischer in \cite{F}.
Basic examples of such groups are the following:
symmetric group $\Sym(m)$ with $m=2$ or $m\geq5$ and $D$ being a set of transpositions;
symplectic group $\Sp_{2m}(2)$, where $m\geq3$ and $D$ is the set of symplectic transvections;
unitary group $\SU_m(2)$, where $m\geq4$ and $D$ is the set of unitary transvections;
orthogonal group $\Ort_{2m}^\epsilon(2)$, where
$D$ is the set of orthogonal transvections,
$m\geq3$, and either $\epsilon=+$ if the Witt index equals $m$ or $\epsilon=-$ if the Witt index equals $m-1$; five groups of sporadic type (in notation of \cite{CH95}): $\Fi_{22}$, $\Fi_{23}$, $\Fi_{24}$, $\Pom_8^+(2):\Sym(3)$, $\Pom_8^+(3):\Sym(3)$. There are two more infinite series of 3-transposition groups in Fischer's classification paper: $\Omega_m^\pm(3)$, where $m\geq5$. Consider an orthogonal group $\Ort^\epsilon_{2m}(3)$ corresponding to a symmetric bilinear form $b(\cdot,\cdot)$ over a field of order three, where $\epsilon$ is defined as above depending on the Witt index.
The group $^+\Omega^\epsilon_{2m}(3)$ is then the subgroup of $\Ort^\epsilon_{2m}(3)$ generated by the 3-transposition conjugacy class $D^+$ of all reflections $d=\sigma_x$ with centers $x$ having $b(x,x)=1$. The corresponding odd degree group $^+\Omega^\epsilon_{2m-1}(3)$ is found within
$^+\Omega^\epsilon_{2m}(3)$ as $\langle D_d\rangle$, where $D_d=C_{D^+}(d)\setminus\{d\}$ for an arbitrary 3-transposition $d\in D^+$. In what follows, we do not need explicit group constructions, but only some properties of their diagrams.

Cuypers and Hall extended Fischer's classification in~\cite{HS95} by dropping the assumptions $O_2(G)O_3(G)\leq Z(G)$ and finiteness of $G$. As a consequence, they showed that every 3-transposition group is locally finite, i.e every finite subset of the group generates a finite subgroup.
Naturally, the groups in the general classification are extensions of the groups obtained by Fischer. For the connected 3-transposition group $(G,D)$, we write $p^{\bullet h}$ with $p\in\{2,3\}$, for a normal $p$-subgroup $N$ with $|D\cap dN|=p^h$ for all $d\in D$. We give a simplified formulation of the classification which is taken from~\cite{Spectra21} and sufficient for our purposes.

\begin{theorem}{\em (Cuypers--Hall Classification Theorem)\cite[Theorem~5.3]{Spectra21}}\label{t:classification}.
Let $(G,D)$ be a finite connected $3$-transposition group. Then, for integral $m$ and $h$,
the group $G$ has one of the central types below. Furthermore, for each $G$ the generating class $D$
is uniquely determined up to an automorphism of $G$.

{\bf PR1.} $3^{\bullet h}:\Sym(2)$, all $h\geq1$;

{\bf PR2(a).} $2^{\bullet h}:\Sym(m)$, all $h\geq0$, all $m\geq4$;

{\bf PR2(b).} $3^{\bullet h}:\Sym(m)$, all $h\geq1$, all $m\geq4$;

{\bf PR2(c).} $3^{\bullet h}:2^{\bullet1}:\Sym(m)$, all $h\geq1$, all
$m\geq4$;

{\bf PR2(d).} $4^{\bullet h}:3^{\bullet1}:\Sym(m)$, all $h\geq1$, all $m\geq4$;

{\bf PR3.} $2^{\bullet h}:\Ort_{2m}^\epsilon(2)$, $\epsilon=\pm$, all $h\geq0$, all $m\geq3$,
$(m,\epsilon)\neq(3,+)$;

{\bf PR4.} $2^{\bullet h}:\Sp_{2m}(2)$, all $h\geq0$, all $m\geq3$;

{\bf PR5.} $3^{\bullet h}{}^+\Omega_{m}^\epsilon(3)$, $\epsilon=\pm$, all $h\geq0$, all $m\geq5$;

{\bf PR6.} $4^{\bullet h}\SU_m(2)'$, all $h\geq0$, all $m\geq3$;

{\bf PR7(a-e).} $\Fi_{22}$, $\Fi_{23}$, $\Fi_{24}$, $\Pom^+_8(2):\Sym(3)$, $\Pom^+_8(3):\Sym(3)$;

{\bf PR8.} $4^{\bullet h}:(3\cdot{}^+\Omega_6^-(3))$, all $h\geq1$;

{\bf PR9.} $3^{\bullet h}:(2\times \Sp_6(2))$, all $h\geq1$;

{\bf PR10.} $3^{\bullet h}:(2\cdot \Ort_8^+(2))$, all $h\geq1$;

{\bf PR11.} $3^{\bullet 2h}:(2\times \SU_5(2))$, all $h\geq1$;

{\bf PR12.} $3^{\bullet 2h}:4^{\bullet1}:\SU_3(2)'$, all $h\geq1$.
\end{theorem}
\begin{rmk}
The notation {\bf PRk} comes from~{\em\cite{Spectra21}},
here $P$ means {\bf P}arabolic and $R$ means {\bf R}eflections.
These abbreviations reflect how the groups arise in the classification.
\end{rmk}

In Theorem~\ref{t:classification}, we follow notation from~\cite{Spectra21}, in particular  $A:B$ means a split group extension with normal subgroup $A$, while $A\cdot B$ is a nonsplit group extension with normal subgroup $A$ and quotient $B$. We write $AB$ indicating that $A$ is a normal subgroup while $B$ is the quotient, but the extension may or may not be split.


Let $V$ be a nonempty set and $(V)$ a graph with $V$ as vertex set.
The $(0,1)$-adjacency matrix of the graph will be denoted $AMat((V))$, and the spectrum of the graph
is the (ordered) spectrum of $AMat((X))$:
$Spec((X))=((\ldots,r_i,\ldots))$.

Suppose that $(G,D)$ is a connected 3-transposition group.
Hall and Shpectorov found in \cite{Spectra21} spectrum of the diagram $(D)$ in all cases of Theorem~\ref{t:classification}. Before formulating their result, it is necessary to introduce some notation and conventions.

Clearly, the all-one vector 1 is an eigenvector of $AMat((V))$ with eigenvalue $k$ if and only if $(V)$ is regular of
degree $k$. If $(V)$ is connected, then the Perron--Frobenius Theorem
implies that $k$ is the largest eigenvalue and the corresponding
eigenspace has dimension one. Following~\cite{Spectra21}, we list $k$ first in the spectrum and separate it from the rest of eigenvalues by a semicolon. We use the convention that $[t]^c$ indicates an eigenvalue $t$ of multiplicity $c$ and $[t]^\ast$ means that the eigenvalue $t$ has multiplicity such that the total multiplicity of all eigenvalues is equal to the size of $V$.

\begin{theorem}\label{t:spectrum}{\em\cite{Spectra21}}
Let $(G,D)$ be a finite 3-transposition group from the conclusion of
Theorem~\ref{t:classification}.
Then the size of $(D)$ and its spectrum are as in the second and third columns of Table~\ref{tab:spectra}, respectively.
\end{theorem}

\begin{table}\caption{Spectra of diagrams}\label{tab:spectra}
\begin{tabular}{| c | c| c|}
\hline
Label & Size & Spectrum \\ \hline
{\bf PR1} & $3^h$ &$((3^h-1;[-1]^{3^h-1}))$ \\
{\bf PR2(a)} & $2^{h-1}m(m-1) $ & $((2^{h+1}(m-2);[2^h(m-4)]^{m-1},[0]^\ast,$ \\
              & & $[-2^{h+1}]^{m(m-3)/2} ))$\\
{\bf PR2(b)} & $3^hm(m-1)/2$ & $((3^h(2m-3)-1; [3^h(m-3)-1]^{m-1}, [-1]^\ast,$ \\
              & & $[-3^h-1]^{m(m-3)/2} ))$ \\
{\bf PR2(c)} & $3^hm(m-1)$ & $((3^h(4m-7)-1; [3^h(2m-7)-1]^{m-1}, [3^h-1]^{m(m-1)/2}, $ \\
             & & $[-1]^\ast, [-3^{h+1}-1]^{m(m-3)/2} ))$ \\
{\bf PR2(d)} & $3(2^{2h-1})m(m-1)$ & $((4^h(6m-10); [4^h(3m-10)]^{m-1}, [0]^\ast, $ \\
              &  & $[-4^h]^{m(m-1)}, [-4^{h+1}]^{m(m-3)/2}))$ \\
{\bf PR3} $\epsilon=+$ & $2^h(2^{2m-1}-2^{m-1})$ & $((2^h(2^{2m-2}-2^{m-1}); [2^{h+m-1}]^{(2^m-1)(2^{m-1}-1)/3},$ \\
             & & $ [0]^\ast,[-2^{h+m-2}]^{(2^{2m}-4)/3}))$ \\
$\epsilon=-$ & $2^h(2^{2m-1}+2^{m-1})$ & $((2^h(2^{2m-2}+2^{m-1}); [2^{h+m-2}]^{(2^{2m}-4)/3},$ \\
             & & $[0]^\ast,[-2^{h+m-1}]^{(2^m+1)(2^{m-1}+1)/3}))$ \\
{\bf PR4} & $2^h(2^{2m}-1)$ & $((2^{2m-1+h}; [2^{m-1+h}]^{2^{2m-1}-2^{m-1}-1},$ \\
   & & $[0]^\ast, [-2^{h+m-1}]^{2^{2m-1}+2^{m-1}-1}))$ \\
{\bf PR5}  & & \\
odd $m\geq5,$  & $3^h(3^{m-1}-3^{(m-1)/2})/2$ & $((3^h(3^{m-2}-2\cdot3^{(m-3)/2})-1;[3^{(m-3)/2+h}-1]^f,$ \\
$\epsilon=+$ &  & $[-1]^\ast, [-3^{(m-3)/2+h}-1]^g))$  \\
    & & for $f=(3^{m-1}-1)/4$ \\
    & & and $g=(3^{m-1}-1-2(3^{(m-1)/2}+1))/4$ \\
odd $m\geq5,$  & $3^h(3^{m-1}+3^{(m-1)/2})/2$ & $((3^h(3^{m-2}+2\cdot3^{(m-3)/2})-1;[3^{(m-3)/2+h}-1]^f,$ \\
 $\epsilon=-$ &  & $[-1]^\ast, [-3^{(m-3)/2+h}-1]^g))$  \\
   & & for $f=(3^{m-1}-1+2(3^{(m-1)/2}-1))/4$ \\
   & & and $g=(3^{m-1}-1)/4$ \\
even $m\geq6$, & $3^h(3^{m-1}-3^{(m-2)/2})/2$ & $((3^{m-2+h}-1;[3^{(m-4)/2+h}-1]^f,$  \\
$\epsilon=+$  &  & $[-1]^\ast,[-3^{(m-2)/2+h}-1]^g))$ \\
              & & for $f=(3^m-9)/8$ \\
              & & and $g=(3^{m/2}-1)(3^{(m-2)/2}-1)/8$ \\
even $m\geq6$, & $3^h(3^{m-1}+3^{(m-2)/2})/2$ & $((3^{m-2+h}-1;[3^{(m-2)/2+h}-1]^f,$ \\
$\epsilon=-$  & &  $[-1]^\ast,[-3^{(m-4)/2+h}-1]^g))$ \\
               & & for $f=(3^{m/2}+1)(3^{(m-2)/2}+1)/8$ \\
               & & and $g=(3^m-9)/8$ \\
{\bf PR6}  & & \\
even $m\geq 4$  & $4^h(2^{2m-1}+2^{m-1}-1)/3$ & $((2^{2h+2m-3};[2^{2h+m-3}]^f,[0]^\ast,[-2^{2h+m-2}]^g))$ \\
      & &  for $f=8(2^{2m-3}-2^{m-2}-1)/9$ \\
      & & and $g=4(2^{2m-3}+7(2^{m-3})-1)/9$ \\
odd $m\geq 3$  & $4^h(2^{2m-1}-2^{m-1}-1)/3$ & $((2^{2h+2m-3}; [2^{2h+m-2}]^f,[0]^\ast, [-2^{2h+m-3}]^g))$ \\
 & & for $f=4(2^{2m-3}-7(2^{m-3})-1)/9$ \\
 & &  and $g=8(2^{2m-3}+2^{m-2}-1)/9$ \\
{\bf PR7(a)} & $3510$ & $((2816; [8]^{3080}, [-64]^{429} ))$ \\
{\bf PR7(b)} & $31671$ & $((28160; [8]^{30888}, [-352]^{722} ))$ \\
{\bf PR7(c)} & $306936$ & $((275264; [80]^{249458}, [-352]^{57477} ))$ \\
{\bf PR7(d)} & $360$ & $((296; [8]^{105}, [-4]^{252}, [-64]^2))$ \\
{\bf PR7(e)} & $3240$ & $((2888; [8]^{2457}, [-28]^{780}, [-352]^2))$ \\
{\bf PR8} & $126\cdot 4^h$ & $((5\cdot 4^{h+2}; [2^{2h+3}]^{35}, [0]^{\ast}, [-4^{h+1}]^{90}))$ \\
{\bf PR9} & $63\cdot3^h$ & $((11\cdot 3^{h+1}-1; [5\cdot3^{h}-1]^{27}, [-1]^{\ast}, [-3^{h+1}-1]^{35}))$ \\
{\bf PR10} & $120\cdot3^h$ & $((19\cdot 3^{h+1}-1; [3^{h+2}-1]^{35}, [-1]^{\ast}, [-3^{h+1}-1]^{84}))$ \\
{\bf PR11} & $165\cdot3^{2h}$ & $((43\cdot 3^{2h+1}-1; [3^{2h+2}-1]^{44}, [-1]^{\ast}, [-3^{2h+1}-1]^{120}))$ \\
{\bf PR12} & $36\cdot3^{2h}$ & $((11\cdot 3^{2h+1}-1; [3^{2h}-1]^{27}, [-1]^{\ast}, [-3^{2h+1}-1]^{8}))$ \\
 \hline

\end{tabular}
\end{table}

Before finishing this section, we introduce an alternative view of the elements of the set of 3-transpositions. The \emph{Fischer space} of a 3-transposition group
$(G,D)$ is a point-line geometry $\Gamma(G,D)$, whose point set is $D$
and where distinct points $c$ and $d$ are collinear if and only if $|cd|=3$.
Observe that any two collinear points $c$ and $d$ lie in a
unique common line, which consists of $c$, $d$, and the third point
$e=c^d=d^c$. It follows from the definition that the connected components of the Fischer space coincide with the conjugacy classes of $G$ contained in $D$. In particular, the Fischer space is connected if and only if the diagram $(D)$ is connected.

\section{Preliminaries: Jordan and Matsuo algebras}\label{sec:matsuo}

Throughout this section we assume that $\mathbb{F}$ is a field of characteristic not two.
Recall that a commutative $\mathbb{F}$-algebra $J$ is called Jordan if any two of its elements $x$ and $y$ satisfy the identity $(x^2y)x=x^2(yx)$. If $x,y,z$ are three elements in an $\mathbb{F}$-algebra, then their associator is $(x, y, z):=(xy)z-x(yz)$.
The associator is convenient when writing identities, for example the Jordan identity $(x^2y)x-x^2(yx)=0$ can be rewritten as $(x^2,y,x)=0$.
To show that an algebra is Jordan we will use the linearized Jordan identity.
\begin{lem}{\em\cite[Proposition~1.8.5(1)]{MC}}\label{Jordan-id}
Let $\mathbb{F}$ be a field of characteristic not two and three.
Then a commutative $\mathbb{F}$-algebra $J$ is a Jordan algebra if and only if $(xz, y, w) + (zw, y, x) + (wx, y, z) = 0$ for all elements $x, y, z, w$ in $J$.
\end{lem}

Suppose that $A$ is an $\mathbb{F}$-algebra and $a\in A$. For an element $\lambda\in\mathbb{F}$ denote by $A_\lambda(a)$ the $\lambda$-eigenspace of the (left) adjoint operator of $a$: $A_\lambda(a)=\{ b\in A~|~ab=\lambda b \}$.

\begin{lem}{\em (Peirce decomposition)}\label{l:peirce}
Suppose that $e$ is an idempotent in a Jordan algebra $J$.
Then the following statements hold.
\begin{enumerate}
\item[$(i)$] $J=J_1(e)\oplus J_0(e)\oplus J_{1/2}(e)$;
 \item[$(ii)$] $J_1(e)+J_0(e)$ is a subalgebra in $J$ and, moreover,
 $J_1(e)^2\subseteq J_1(e)$, $J_0(e)^2\subseteq J_0(e)$, and $J_1(e)J_0(e)=(0)$;
 \item[$(iii)$] $J_{1/2}(e)^2\subseteq J_0(e)+J_1(e)$ and $J_{1/2}(e)(J_0(e)+J_1(e))\subseteq J_{1/2}(e)$.
\end{enumerate}

\end{lem}

Suppose that $\eta\in\mathbb{F}$ and $\eta\neq0,1$.
Fix a 3-transposition group $(G,D)$. The \emph{Matsuo algebra} $M_\eta(G,D)$ over $\mathbb{F}$, corresponding to $(G,D)$ and $\eta$, has the point set $D$ as its basis. Multiplication is defined
on $D$ as follows:
$$c\cdot d=\left\{
\begin{array}{rl}
c,&\mbox{if }c=d;\\
0,&\mbox{if }|cd|=2;\\
\frac{\eta}{2}(c+d-e),&\mbox{if }|cd|=3\mbox{ and }e=c^d=d^c.
\end{array}
\right.$$
We use the dot for the algebra product to distinguish it from
the multiplication in the group $G$.
Matsuo algebras $M_\eta(G,D)$ generalize Jordan algebras in the following way:
the assertions of Lemma~\ref{l:peirce} hold for any $e\in D$ if we replace $J$ with $M_\eta(G,D)$ and $\frac{1}{2}$ with $\eta$. This means that Matsuo algebras are examples of axial algebras of Jordan type $\eta$ (see \cite[Theorem 6.2]{hrs} for details).

The Matsuo algebra $M=M_\eta(G,D)$ admits a bilinear symmetric form $(\cdot,\cdot)$ that associates with the algebra product, i.e. $(u\cdot v,w)=(u,v\cdot w)$ for arbitrary algebra elements $u$, $v$, and $w$.
This form is given on the basis $D$  by the following:
$$
(c,d)=\left\{
\begin{array}{rl}
1,&\mbox{ if }c=d;\\
0,&\mbox{ if }|cd|=2;\\
\frac{\eta}{2},&\mbox{ if }|cd|=3.
\end{array}\right.
$$

The radical $M^\perp$ of the form is the set of elements orthogonal to $M$:
$$M^\perp=\{u\in M\mid (u,v)=0\mbox{ for all }v\in M\}.$$
Since the form associates with the algebra product, $M^\perp$ is an ideal in $M$.
It turns out that in many situations the radical includes all proper ideals of $M$.
\begin{prop}[\cite{kms}]\label{p:summary}
Suppose that $M=M_\eta(G,D)$ is a Matsuo algebra.
If the diagram $(D)$ is connected, then each ideal of $M$ lies in the radical $M^\perp$.
\end{prop}

Following \cite{hrs}, we call this form {\it the Frobenius form} of $M_\eta(G,D)$.

\begin{lem}\footnote{This lemma was mentioned by S.~Shpectorov in his talk at Axial seminar, 12/10/21, \url{https://sites.google.com/view/axial-algebras/home}}\label{l:eigenvalues} Let $\mathbb{F}$ be a field of characteristic $0$ and $\eta\in\mathbb{F}\setminus\{0,1\}$. Suppose that $(G,D)$ is a $3$-transposition group and $M=M_\eta(G,D)$ is the Matsuo algebra for $(G,D)$. Fix some order of elements of $D$ and denote by
$\mathcal{M}$ the Gram matrix of the Frobenius form of $M$ with respect to $D$
and by $\mathcal{A}$ the adjacency matrix $AMat((D))$. Then $\zeta$ is an eigenvalue of $\mathcal{A}$ with multiplicity $k$ if and only if $1+\frac{\eta}{2}\zeta$ is an eigenvalue of $\mathcal{M}$
with multiplicity~$k$.
\end{lem}
\begin{proof}
It follows from the definitions of $\mathcal{M}$ and $\mathcal{A}$
that $\mathcal{M}=I+\frac{\eta}{2}\mathcal{A}$.
Write $D=\bigcup\limits_{i=1}^kD_i$ as a disjoint union of $k$ conjugacy classes $D_i$ and denote $G_i=\langle D_i\rangle$, where $1\leq i\leq k$. By \cite[Theorem~6.3]{hrs}, $M$ is isomorphic to the direct product of $M_\eta(G_i,D_i)$. We can assume that the elements in $D$ are arranged by classes, first the class $D_1$, then $D_2$, and so on. Then the matrices $\mathcal{M}$ and $\mathcal{A}$ are block-diagonal. Therefore, we can assume that $(D)$ is connected and hence
$D$ is a single conjugacy class of 3-transpositions.
It follows from the description of the spectrum of $(D)$ in Theorem~\ref{t:spectrum} that all eigenvalues in $Spec((D))$ are rational, so we can find a basis comprising of eigenvectors whose all coordinates are rational. If $T$ is the transformation matrix from $D$ to the new basis, then $T^{-1}\mathcal{A}T$ is a diagonal matrix. Hence
$T^{-1}\mathcal{M}T=I+\frac{\eta}{2}T^{-1}\mathcal{A}T$ is diagonal and the result follows.
\end{proof}

\begin{cor}\label{c:critical_value}
Let $\mathcal{M}$ and $\mathcal{A}$ be as in Lemma~{\em\ref{l:eigenvalues}}.
If $\eta=\frac{1}{2}$, then the multiplicity of $0$ in the spectrum of
$\mathcal{M}$ is equal to that of $-4$ in the spectrum of $\mathcal{A}$.
\end{cor}
\begin{proof} From the bijection between eigenvalues of
$\mathcal{M}$ and $\mathcal{A}$ in Lemma~\ref{l:eigenvalues}, we find that $0$ corresponds to $\zeta$ such that $1+\frac{1}{4}\zeta=0$, that is $\zeta=-4$.
\end{proof}

De Medts and Rehren classified Matsuo algebras that are Jordan algebras in \cite{TR2017}. Yabe corrected a gap in the case when the characteristic of the field equals 3~\cite{Yabe}.
For simplicity and since we are mainly interested in characteristic zero, we state the result when the field characteristic is not three.
\begin{theorem}{\em\cite[Main Theorem]{TR2017}}\label{t:tom}
Let $\mathbb{F}$ be a field, $\operatorname{char}(\mathbb{F})\neq2,3$, and let $J$ be a Jordan algebra over $\mathbb{F}$ which is also a Matsuo algebra. Then $J$ is a direct product of Matsuo algebras $J_i=M_{1/2}(G_i, D_i)$ corresponding to $3$-transposition groups $(G_i, D_i)$,
where for each $i$,
\\(i) either $G_i=\Sym(n)$, and $J_i$ is the Jordan algebra of $n\times n$ symmetric matrices over $F$ with zero row sums;
\\(ii) $G_i\simeq 3^2:2$, and $J_i$ is the Jordan algebra of hermitian $3\times 3$ matrices over
the quadratic \'etale extension $E=\mathbb{F}[x]/(x^2 + 3)$.
\end{theorem}

\section{Preliminaries: wreath product}\label{sec:wreath}

In this section, we discuss 3-transposition groups that correspond to types {\bf PR2(a-e)} in Theorem~\ref{t:classification}. All these groups can be constructed from the wreath product of a group whose elements have orders not exceeding 3 and a symmetric group.

Denote the base group of the wreath product $G=T\operatorname{wr}\Sym(n)$ by $B$, i.e. $B=T^n$.  The natural injection $\iota_i$ of $T$ as the $i$-th direct factor $T_i$ of $B$ is given by $\iota_i(t)=t_i$, where $1\leq i\leq n$. The projection $\pi_i$ of $B$ onto $T$
induced by the $i$-th factor is given by $\pi_i(b)=b(i)$. We identify $\Sym(n)$ with the complement to $B$ in $G$ which acts naturally on the indices from $\{1,\ldots,n\}$. Let $Wr(T,n)$ be the subgroup $\langle d^G\rangle$ of $G$, where $d$ is a transposition of the complement to $B$. Note that the factor group $G/Wr(T,n)$ is isomorphic to the abelian group $T/T'$, in particular $Wr(T,n)$ can be a proper subgroup of $G$. The following statement describes when $d^G$ is a class of 3-transpositions.
\begin{prop}{\em\cite[Theorem~6]{Zara88},
\cite[Prop.~8.1]{H93}}\label{p:w(k,n)}
Suppose that $T$ is a finite group and $G=T\wre\Sym(n)$. Fix a transposition $d$ of $\Sym(n)$. Then $d^G$ is a class of $3$-transpositions if and only if each element of $T$ has order $1$, $2$, or~$3$.
\end{prop}

Note that the groups $T$ with restrictions as in the proposition were classified in \cite{Neumann}.
The next two lemmas are well known and describe how we deal with points and lines of the Fischer space of $Wr(T,n)$.

\begin{lem}{\em\cite[Lemma~3.2]{ms}}\label{l:3trans}
Consider the wreath product  $G=T\wre{\Sym(n)}$ and a transposition $d\in\Sym(n)$. Then $d^G$ consists of elements $t_it^{-1}_j(i,j)$, where $t\in T$ and $1\leq i< j\leq n$.
\end{lem}

\noindent{\bf Notation.} We write $t.(i,j)$ for the 3-transposition $t_it_j^{-1}(i,j)$ from Lemma~\ref{l:3trans}. Since $t.(i,j)=t^{-1}.(j,i)$, we will usually assume that $i<j$.

\begin{lem}{\em\cite[Lemma~3.3]{ms}}\label{l:lines} Suppose that each element of $T$ has order $1$, $2$, or $3$.
Then each line of the Fischer space of $Wr(T,n)$ coincides with one of the following sets.
\begin{enumerate}
\item[$(i)$] $\{t.(i,j), s.(j,k), ts.(i,k)\}$, where $s,t\in T$ and $1\leq i<j<k\leq n$;
\item[$(ii)$] $\{t.(i,j), s.(i,j), st^{-1}s.(i,j)\}$, where $s,t\in T$, $|st^{-1}|=3$, and $1\leq i<j\leq n$.
\end{enumerate}
\end{lem}

Now we focus on 3-transposition groups $Wr(p,n)$,
where $p$ means the cyclic group of order $p\in\{2,3\}$.
Following \cite{ms} and \cite{gjmss},
we will use the following descriptions of the Fischer spaces of these groups.
Let $n$ be an integer and $n\geq3$. For $p\in\{2, 3\}$ consider the
$n$-dimensional permutational module $V$ of $\Sym(n)$ over $\mathbb{F}_p$.
Let $e_i$, $i\in\{1,\ldots,n\}$, be a basis of $V$ permuted by $\Sym(n)$.
Then the natural semi-direct product $V\rtimes\Sym(n)$ is isomorphic to $p\wre\Sym(n)$.
Denote the $(n-1)$-dimensional ‘sum-zero’ submodule of $V$ by $U$. Then $Wr(p, n)$ is isomorphic to the natural semidirect product $U\rtimes\Sym(n)$. Note that, for $p=2$ and even $n$, $U$ contains a 1-dimensional ‘all-one’ submodule, which is the center of $Wr(2, n)$.
When $p=3$, $U$ is irreducible. In both cases, $U$ is the unique minimal non-central normal subgroup of $Wr(p, n)$ and $Wr(p, n)/U\simeq\Sym(n)$.
Since $\Sym(n)$ does not have proper factor groups containing commuting involutions, we conclude
that, up to the center, groups $Wr(p,n)$ have no other factors that are 3-transposition groups.
Now we describe the Fischer spaces of these groups.

Assume that $p=2$. It follows from Lemmas~\ref{l:3trans} and~\ref{l:lines} that the Fischer space of
$Wr(2, n)=U:Sym(n)$ consists of $n(n-1)$ points: $b_{i,j}=(i,j)$ and $c_{i,j}=(e_i+e_j)(i,j)$, for $1\leq
i<j\leq n$; and $n^2$ lines, where each `b' line $\{b_{i,j},b_{i,k},b_{j,k}\}$, $1\leq i<j<k\leq n$, is complemented by three `bc' lines $\{b_{i,j},c_{i,k},c_{j,k}\}$, $\{b_{i,k},c_{i,j},c_{j,k}\}$,
and $\{b_{j,k},c_{i,j},c_{i,k}\}$.

Assume that $p=3$. By Lemma~\ref{l:3trans}, for each pair $i$ and $j$ with $1\leq i<j\leq n$ we have three points: $b_{i,j}=(i,j)=b_{j,i}$, $c_{i,j}=(e_i-e_j)(i,j)$ and $c_{j,i}=(e_j-e_i)(i,j)$. Consequently, the Fischer space has $\frac{3n(n-1)}{2}$ points.
By Lemma~\ref{l:lines}, the lines are of several types. First, for each $1\leq i<j\leq n$, the triple (1) $\{b_{i,j},c_{i,j},c_{j,i}\}$ is a line. Secondly, for distinct $i$, $j$, and $k$ in $\{1,\ldots,n\}$, the triples (2) $\{b_{i,j},b_{i,k},b_{j,k}\}$, (3) $\{b_{i,j},c_{i,k},c_{j,k}\}$, (4) $\{b_{j,k},c_{i,j},c_{i,k}\}$, and (5) $\{c_{i,j},c_{j,k},c_{k,i}\}$ are lines.

Using the descriptions of Fischer spaces, we find bases for the corresponding Matsuo algebras.
\begin{lem}\label{l:radical} Let $G=Wr(p,n)$, where $p\in\{2,3\}$ and $n\geq4$.
Denote by $D$ the corresponding 3-transposition set and by $M$ the Matsuo algebra $M_{1/2}(G,D)$. Then $\dim M^\perp=\frac{n(n-3)}{2}$ and the following assertions hold.
\begin{enumerate}
\item[$(i)$] If $p=2$, then $M^\perp$ is the span of elements $$b_{i,j}-b_{i,l}-b_{j,k}+b_{k,l}+c_{i,j}-c_{i,l}-c_{j,k}+c_{k,l},$$
where $i,j,k,l$ are distinct elements of $\{1,\ldots,n\}$ and $i$ is less than $j,k,l$.
\item[$(ii)$] If $p=3$, then $M^\perp$ is the span of elements
$$b_{i,j}-b_{i,l}-b_{j,k}+b_{k,l}+c_{i,j}-c_{i,l}-c_{j,k}+c_{k,l}+c_{j,i}-c_{l,i}-c_{k,j}+c_{l,k},$$
where $i,j,k,l$ are distinct elements of $\{1,\ldots,n\}$ and $i$ is less than $j,k,l$.
\end{enumerate}
\end{lem}
\begin{proof}
By Corollary~\ref{c:critical_value}, the dimension of $M^\perp$ is equal to the multiplicity of $-4$
in the spectrum of the diagram $(D)$. According to \cite[Example~PR2]{CH95}, if $p=2$, then
$G$ corresponds to the type {\bf PR2(a)} in Theorem~\ref{t:classification}, while if $p=3$, then $G$ corresponds to the type {\bf PR2(b)}. In both cases the parameter $h$ equals 1. It follows from
Theorem~\ref{t:spectrum} that the $-4$ has the multiplicity $\frac{n(n-3)}{2}$ in $Spec((D))$.
This implies that $\dim M^\perp=\frac{n(n-3)}{2}$.

For arbitrary distinct integers $i,j,k,l$ such that $1\leq i,j,k,l\leq n$
and $i$ less than $j,k,l$ denote
$$r(i,j)(k,l)=b_{i,j}-b_{i,l}-b_{j,k}+b_{k,l}+c_{i,j}-c_{i,l}-c_{j,k}+c_{k,l}\text{ if }p=2,$$
and
$$r(i,j)(k,l)=b_{i,j}-b_{i,l}-b_{j,k}+b_{k,l}+c_{i,j}-c_{i,l}-c_{j,k}+c_{k,l}+c_{j,i}-c_{l,i}-c_{k,j}+c_{l,k} \text{ if }p=3.$$
We claim that each $r(i,j)(k,l)$ belongs to $M^\perp$.
By symmetry of indices, it suffices to show this for $r(1,2)(3,4)$. Now we verify that each 3-transposition of $D$ is orthogonal to $r(1,2)(3,4)$ with respect to the Frobenius form.
Suppose that $p=3$. Take a 3-transposition $x_{ij}\in D$, where $x\in\{b,c\}$.
First we consider the case $i,j\in\{1,2,3,4\}$. If $x_{i,j}\in\{b_{1,2}, c_{1,2}, c_{2,1}\}$,
then
\begin{multline*}
(x_{i,j}, b_{1,2}+c_{1,2}+c_{2,1})=1+\frac{1}{4}+\frac{1}{4}=\frac{3}{2}, (x_{i,j}, b_{3,4}+c_{3,4}+c_{4,3})=0,\\ (x_{i,j}, -b_{1,4}-b_{2,3}-c_{1,4}-c_{4,1}-c_{3,4}-c_{4,3})=-6\cdot\frac{1}{4}=-\frac{3}{2}.
\end{multline*}
Therefore, we infer that $(x_{i,j}, r(1,2)(3,4))=0$. Similarly, we see that $(x_{i,j}, r(1,2)(3,4))=0$ when  $x_{ij}\in\{b_{3,4},c_{3,4},c_{4,3},b_{1,4},c_{1,4},c_{4,1},b_{2,3},c_{2,3},c_{3,2}\}$. Let $x_{i,j}\in\{b_{1,3},c_{3,1},c_{3,1}, b_{2,4},c_{2,4},c_{4,2}\}$. Then
$$(x_{i,j}, b_{1,2}+c_{1,2}+c_{2,1})=(x_{i,j}, b_{3,4}+c_{3,4}+c_{4,3})=\frac{3}{4}, (x_{i,j}, -b_{1,4}-c_{1,4}-c_{4,1})=(x_{i,j}, b_{2,3}+c_{2,3}+c_{3,2})=-\frac{3}{4}.$$
Therefore, we see that $(x_{i,j}, r(1,2)(3,4))=0$. Clearly, if $i,j\not\in\{1,2,3,4\}$, then $(x_{i,j}, r(1,2)(3,4))=0$.
So it remains to consider the case when $|\{i,j\}\cap\{1,2,3,4\}|=1$.
Note that for each integer $k\in\{1,2,3,4\}$, exactly six out of the twelve terms in $r(1,2)(3,4)$ contain $k$ as an index, moreover, three of these six are included in the expression with a plus sign and three with a minus sign. This implies that $x_{ij}$ is orthogonal to $r(1,2)(3,4)$.
The case $p=2$ can be considered in a similar manner.

Now we present $\frac{n(n-3)}{2}$ linearly independent elements among $\{r(i,j)(k,l)\}$.
Consider two sets of elements of $D$: $\mathcal{B}_1=\{r(i,j)(n-1,n) \mid 1\leq i<j< n-1 \}$ and  $\mathcal{B}_2=\{r(1,n-1)(i,n) \mid i\neq 1,n-1,n \}$.
Suppose that the set $\mathcal{B}_1\cup\mathcal{B}_2$ is linearly dependent in $M$.
Note that if $(i,j)$ is a pair with $1\leq i<j<n-1$, then
$r(i,j)(n-1,n)$ is the only element of $\mathcal{B}_1\cup\mathcal{B}_2$ including $b_{i,j}$ in its expression. It follows that if a non-trivial linear combination of elements of $\mathcal{B}_1\cup\mathcal{B}_2$ is equal to 0, then only elements from $\mathcal{B}_2$ have non-zero coefficients.
On the other hand, if $i\neq1,n-1,n$, then $r(1,n-1)(i,n)$ is the only element in $\mathcal{B}_2$
including $b_{in}$ in its expression and hence  $\mathcal{B}_2$ is linearly independent;
we arrive at a contradiction. Thus, the set
$\mathcal{B}_1\cup\mathcal{B}_2$ is linearly independent. Since $|\mathcal{B}_1|=\frac{(n-2)(n-3)}{2}$, $|\mathcal{B}_2|=n-3$, and $\mathcal{B}_1\cap\mathcal{B}_2=\varnothing$,
we find $\frac{n(n-3)}{2}$ linearly independent elements in $M^\perp$.
This implies that the set $\{r(i,j)(k,l)\}$ spans the radical of $M$ and as a basis we can take elements of $\mathcal{B}_1\cup\mathcal{B}_2$.
\end{proof}

\section{Proof of the main theorem}\label{sec:proof}

In this section, we prove Theorem~\ref{t:main}.
Throughout, we suppose that $\mathbb{F}$ is a field of characteristic zero.
First we consider the case when the parameter $\eta$ in Matsuo algebra is not equal to $\frac{1}{2}$.

\begin{lem} Suppose that $\eta\neq\frac{1}{2}$ and $M=M_\eta(G,D)$ is the Matsuo algebra for a finite connected $3$-transposition group $(G,D)$. A factor of $M$ by its ideal $I\neq M$ is a Jordan algebra if and only if one of the following statements holds.
\begin{enumerate}
    \item[(i)] $G$ is the cyclic group of order $2$ and $I=(0)$;
    \item[(ii)] $\eta=2$, the product of every two distinct elements of $D$ has order $3$, $I$ is the span of elements $d-e$, where $d$ and $e$ run over $D$. In this case $M/I$ is one-dimensional.
 \end{enumerate}
\end{lem}
\begin{proof}
Clearly, if $D=\{d\}$, then $G$ is the cyclic group of order 2 and $M$ is generated by $d$.
So $M$ is associative and one-dimensional. Therefore, we can assume that $|D|\geq2$.

Take any $c\in D$. Since $(D)$ is connected and $|D|\geq2$,
there exists $d\in D$ such that $|cd|=3$. If $x\in M$, then denote by $\overline{x}$ the image of $x$ in $M/I$. Note that $\overline{c}$ is an idempotent in $M/I$. Denote by $e$ the third point on the line through $c$ and $d$ in the Fischer space $\Gamma(G,D)$. Then $e\cdot(c-d)=\frac{\eta}{2}(e+c-d-e-d+c)=\eta(c-d)$ and hence $\overline{e}\cdot(\overline{c}-\overline{d})=\eta(\overline{c}-\overline{d})$.
It follows from Lemma~\ref{l:peirce} that $\overline{c}=\overline{d}$. Since $c$ is an arbitrary element of $D$ and $(D)$ is connected, we infer that $M/I$ is one-dimensional and spanned by $\overline{d}$ for each $d\in D$. So $M/I$ is associative and hence it is a Jordan algebra.
Suppose that there exist $d,e\in D$ such that $|de|=2$. Then $0=\overline{d\cdot e}=\overline{d}\cdot\overline{e}=\overline{d}^2=\overline{d}$ and hence $d\in I$. All elements of $D$ are conjugated in $G$, so this is true for all elements of $D$; a contradiction. It remains to show that $\eta=2$.
Therefore, the product of any two distinct elements of $D$ has order 3. Suppose that $c$ and $d$ are distinct elements in $D$. By \cite{kms}, $I\subseteq M^\perp$ and hence $c-d\in M^\perp$.
On the other hand, $(c,c-d)=1-\frac{\eta}{2}$ and hence $\eta=2$.

Conversely, suppose that $\eta=2$ and the product of any two elements in $D$ has order 3.
We show that for every $c,d\in D$ it is true that $c-d\in M^\perp$.
First, we see that $(c,c-d)=(d,c-d)=1-1=0$. If $e\in D\setminus\{c,d\}$, then $(e,c)=(e,d)=1$ and hence $(e,c-d)=0$. Since $c-d$ is orthogonal to all elements in $D$ with respect to the Frobenius form,
we infer that $c-d\in M^\perp$. This implies that $M/M^\perp$ is 1-dimensional and the result follows.
\end{proof}

Matsuo algebras $M_{1/2}(G,D)$, where $(G,D)$ is a 3-transposition groups,
that are Jordan algebras were classified in \cite{TR2017}. In particular,
if $(D)$ is connected, then $G\simeq\Sym(n)$ or has the same central type as the Frobenius group $3^2:2$. In view of Theorem~\ref{t:classification}, the symmetric group has the type ${\bf PR2(a)}$ and $3^2:2$ has the type ${\bf PR1}$. It follows from Corollary~\ref{c:critical_value} and Theorem~\ref{t:spectrum} that
$M_{1/2}(G,D)$ are simple in these cases. To prove Theorem~\ref{t:main} it remains to consider Matsuo algebras for $\eta=\frac{1}{2}$ whose radical is nontrivial.

\begin{prop}\label{p:main}
Suppose that $(G,D)$ is a finite connected $3$-transposition group and the Matsuo algebra
$M=M_{1/2}(G,D)$ has nontrivial radical $M^\perp$. Then $J=M/M^\perp$ is a Jordan algebra if and only if one the following statements holds.
\begin{enumerate}
\item $G\simeq 2^{\bullet1}:\Sym(m)$, where $m\geq4$ and $\dim J=\frac{m(m+1)}{2}$; 
\item $G\simeq 3^{\bullet1}:\Sym(m)$, where $m\geq4$ and $\dim J=m^2$;  
\item $G\simeq O_8^+(2)$ and $\dim J=36$;   
\item $G\simeq O_6^-(2)\simeq{}^+\Omega_5^+(3)$ and $\dim J=21$;   
\item $G\simeq Sp_6(2)$ and $\dim J=28$;  
\item $G\simeq{}^+\Omega_6^-(3)$ and $\dim J=36$; 
\item $G\simeq 2\times SU_4(2)\simeq{}^+\Omega_5^-(3)$ and $\dim J=25$; 
\item $G\simeq SU_5(2)$ and $\dim J=45$; 
\item $G\simeq 4^{\bullet1}SU_3(2)'$ and $\dim J=28$; 
\end{enumerate}
\end{prop}
\begin{proof}
We sort out possibilities for $G$ from Theorem~\ref{t:classification}.
By Corollary~\ref{c:critical_value}, the dimension of $M^\perp$ equals to the multiplicity of $-4$ in the spectrum of the diagram $(D)$. Therefore, we need to find all $G$ such that $-4$ is in the spectrum of $(D)$. According to Table~\ref{tab:spectra}, $G$ does not belong to types $\bf{PR1,PR2(c),PR7(a,b,c,e), PR8-PR12}$. Now we consider the remaining cases.

Assume that the type of $G$ is ${\bf PR2(a)}$. According to Table~\ref{tab:spectra}, we see that $-2^{h+1}=-4$ and hence $h=1$.
Therefore, $G=2^{\bullet1}:\Sym(m)\simeq Wr(2,m)$, $|D|=m(m-1)$, $\dim M^\perp=\frac{m(m-3)}{2}$, and $\dim J=\frac{m(m-3)}{2}$. We claim that $J$ is a Jordan algebra in this case.
By Lemma~\ref{Jordan-id}, this is true if and only if every $a,b,c,d\in D$ satisfy the following:
$$w(a,b,c,d)=(a\cdot d, b, c) + (d\cdot c, b, a) + (c\cdot a, b, d)\in M^\perp.$$
We use the description of $D$ as in Section~\ref{sec:wreath}, so each $a\in D$ is equal to some
$x_{i,j}$, where $1\leq i\neq j\leq m$ and $x\in\{'b','c'\}$. In this notation, expressions for elements $a,b,c,d$ include no more than 8 distinct indices $i,j$, so we can consider $a$, $b$, $c$, and $d$ as elements of $H_k=Wr(2,k)$ with $k\leq 8$ after renumbering indices in the corresponding elements $x_{i,j}$. Using GAP\footnote{all verifications in GAP related to this proof are contained in the following file: \url{https://github.com/AlexeyStaroletov/AxialAlgebras/blob/master/JordanFactors/Groups}}~\cite{GAP}, we verify that the element $w(a,b,c,d)$ for every 3-transpositions $a,b,c,d$ from $H_k$ lies in the radical of the Frobenius form of the Matsuo algebra for $H_k$, where $4\leq k\leq 8$. Note that the following enlargement property is true for the radical in these cases: elements from Lemma~\ref{l:radical} that span $M_k^\perp$ belong to $M_n^\perp$ for all $n\geq k$. This implies that $w(a,b,c,d)\in M^\perp$ for all $a,b,c,d\in D$; as claimed.

Assume that the type of $G$ is ${\bf PR2(b)}$. Then $-3^h-1=-4$ and hence $h=1$. So
$|D|=\frac{3m(m-1)}{2}$,  $G=3^{\bullet 1}:\Sym(m)\simeq Wr(3,m)$ and $\dim M^\perp=\frac{m(m-3)}{2}$.
So $\dim J=\frac{3m(m-1)}{2}-\frac{m(m-3)}{2}=m^2$. We verify that $J$ is a Jordan algebra in the same way as in the previous case. Namely, we use GAP to verify the linearized Jordan identity from Lemma~\ref{Jordan-id} for all $m$ with $4\leq m\leq 8$. The general case follows from the description of a basis of $M^\perp$ in Lemma~\ref{l:radical} since this basis satisfy the enlargement property with increasing~$m$.

Assume that the type of $G$ is ${\bf PR2(d)}$. Then $-4^h=-4$ and hence $h=1$. According to \cite[Example~PR2]{CH95}, $G$ has the same central type as $Wr(\Alt(4), m)$. By the wreath product construction, we can assume that $Wr(\Alt(4), m)$ is a subgroup $Wr(\Alt(4), n)$ if $m\leq n$ and hence there is also an embedding of the corresponding Matsuo algebras. Clearly, if a factor of an algebra $A$ by its ideal is a Jordan algebra, then all subalgebras of $A$ also have factors that are Jordan algebras. Using GAP and Lemma~\ref{Jordan-id}, we verify that the factor algebra of the Matsuo algebra for $Wr(Alt(4), 4)$ by its radical is not a Jordan algebra. Therefore,
this case is impossible.

Assume that the type of $G$ is ${\bf PR3}$. Recall that $m\geq3$ and $(m,\epsilon)\neq(3,+)$.
If $\epsilon=+$, then $-2^{h+m-2}=-4$. This implies that $h=0$ and $m=4$. Then $|D|=2^7-2^3=120$, $\dim M^\perp=(2^8-4)/3=84$, and $\dim J=36$. If $\epsilon=-$, then $-2^{h+m-1}=-4$, so
$h=0$ and $m=3$. Therefore, we see that $|D|=2^5+2^2=36$, $\dim M^\perp=(2^3+1)(2^2+1)/3=15$, and $\dim J=21$. Using GAP, we verify that in both cases $J$ is a Jordan algebra.

Assume that the type of $G$ is ${\bf PR4}$. Then $-2^{h+m-1}=-4$. Since $m\geq3$, we infer that $h=0$ and $m=3$. According to Table~\ref{t:spectrum}, we find that $|D|=2^6-1=63$, $\dim M^\perp=2^5+2^2-1=35$, $\dim J=28$. Using GAP, we verify that $J$ is Jordan.

Assume that the type of $G$ is ${\bf PR5}$. If $m$ is odd, then $-3^{(m-3)/2+h}-1=-4$, so $m=5$ and $h=0$. According to \cite[Example~1.5]{CH95}, it is true that ${}^+\Omega_5^-(3)\simeq 2\times SU_4(2)\simeq$ and ${}^+\Omega_5^+(3)\simeq O_6^-(2)$. The algebra $J$ is considered in the corresponding cases for $G\in\{SU_4(2), O_6^-(2)\}$. Suppose that $m$ is even.
According to Table~\ref{t:spectrum}, we see that $\epsilon=-$ and
 $-3^{(m-4)/2+h}-1=-4$. This implies that $m=6$ and $h=0$. Then $|D|=(3^5+3^2)/2=126$, $\dim M^\perp=(3^6-9)/8=90$, and $\dim J=36$.
Using GAP, we verify that $J$ is a Jordan algebra.

Assume that the type of $G$ is ${\bf PR6}$.
If $m$ is even, then $-2^{2h+m-2}=-4$, so $m=4$ and $h=0$. Therefore, $|D|=(2^7-1+2^3)/3=45$, $\dim M^\perp=4(2^5-1+7\cdot2)/9=20$, and hence $\dim J=25$. Using GAP, we see that $J$ is a Jordan algebra.
If $m$ is odd, then either $m=5$ and $h=0$ or $m=3$ and $h=1$. In the first case,
we find that $|D|=(2^9-1-2^4)/3=165$, $\dim M^\perp=8(2^7-1+2^3)/9=120$, and $\dim J=45$.
In the second case, $|D|=4(2^5-1-2^2)/3=36$, $\dim M^\perp=8(2^3-1+2)/9=8$, $\dim J=28$.
Using GAP, we see that $J$ is a Jordan algebra in these cases.

Assume that the type of $G$ is ${\bf PR7(d)}$. In this case, $|D|=360$, $\dim M^\perp=252$, and $\dim J=108$.
We use the defining relations of $G$ from the Appendix of \cite{HS95} to do the calculations with $J$. Using GAP and Lemma~\ref{Jordan-id}, we verify that $J$ is not a Jordan algebra in this case.
\end{proof}

We conclude this section with the following

\begin{Problem}
In each case of Proposition \ref{p:main} find the smallest ideal $I$ such that $M/I$ is Jordan and identify the corresponding Jordan factors.
\end{Problem}

\section{Octonion and Albert algebras}\label{sec:albert}

Throughout this section we suppose that $\mathbb{F}$ is a field of characteristic not two and three.
Recall that an octonion algebra over $\mathbb{F}$ is a composition algebra that has dimension 8 over $\mathbb{F}$. This means that it is a unital non-associative algebra $\mathbb{O}$ over $\mathbb{F}$ with a non-degenerate quadratic form $N$ such that
$N(xy)=N(x)N(y)$ for all $x$ and $y$ in $\mathbb{O}$. For a given field $\mathbb{F}$, there may exist several octonion algebras, but if $\mathbb{F}$ is algebraically closed field, then all octonion algebras over $\mathbb{F}$ are isomorphic. We use the construction of an octonion algebra from \cite[Section 4.3.2]{Wilson}, which is a generalization of the real octonion algebra, also known as the Cayley numbers.

Take 7 mutually orthogonal square roots of $-1$, labeled $i_0,\ldots,i_6$
(with subscripts understood modulo 7), subject
to the condition that for each $t$, the elements $i_t$, $i_{t+1}$, $i_{t+3}$ satisfy the same
multiplication rules as $i$, $j$, and $k$ (respectively) in the quaternion algebra:
$ij=k=-ij, jk=i=-kj, ki=j=-ik$. For convenience, we write all their pairwise products in Table~\ref{t:oct}.

\begin{table}\centering\caption{Octonions}\label{t:oct}
\begin{tabular}{| c | c c c c c c c|}
\hline
$1$ & $i_0$ & $i_1$ & $i_2$ & $i_3$ & $i_4$ & $i_5$ & $i_6$ \\ \hline
$i_0$ & $-1$ & $i_3$ & $i_6$ & $-i_1$ & $i_5$ & $-i_4$ & $-i_2$ \\
$i_1$ & $-i_3$ & $-1$ & $i_4$ & $i_0$ & $-i_2$ & $i_6$ & $-i_5$ \\
$i_2$ & $-i_6$ & $-i_4$ & $-1$ & $i_5$ & $i_1$ & $-i_3$ & $i_0$ \\
$i_3$ & $i_1$ & $-i_0$ & $-i_5$ & $-1$ & $i_6$ & $i_2$ & $-i_4$ \\
$i_4$ & $-i_5$ & $i_2$ & $-i_1$ & $-i_6$ & $-1$ & $i_0$ & $i_3$ \\
$i_5$ & $i_4$ & $-i_6$ & $i_3$ & $-i_2$ & $-i_0$ & $-1$ & $i_1$ \\
$i_6$ & $i_2$ & $i_5$ & $-i_0$ & $i_4$ & $-i_3$ & $-i_1$ & $-1$ \\ \hline
\end{tabular}
\end{table}

Now we define the Albert algebra $A(\mathbb{F})$ corresponding to $\mathbb{O}$.
Elements of $A(\mathbb{F})$ are $3\times 3$ Hermitian matrices (i.e. matrices $x$ such that $x^T=\overline{x}$) over the octonion algebra $\mathbb{O}$. For brevity let us define
$$(d, e, f~|~D, E, F) =
\left(\begin{matrix}
d & F & \overline{E} \\
\overline{F} & e & D \\
E & \overline{D} & f
\end{matrix}
\right),$$
where $d, e, f$ lie in $\mathbb{F}$ and $\overline{\phantom{a}}$ denotes the octonion conjugation, i.e., the linear map fixing 1 and negating $i_n$ for all $n$.
Multiplication of such matrices makes sense, and the Jordan product $X\circ Y=\frac{1}{2}(XY+YX)$
for every $X, Y\in A(\mathbb{F})$ allows to consider $A(\mathbb{F})$ as a simple Jordan algebra.

Fix the following idempotents of the Albert algebra $A(\mathbb{F})$:
$$a = \frac{1}{2}(1,1,0 \mid 0,0,i_0) =\frac{1}{2}
\begin{pmatrix}
 1 & i_0 & 0 \\
-i_0 & 1 & 0 \\
0 & 0 & 0
\end{pmatrix},~
b = \frac{1}{2}(1,0,1 \mid 0,i_1,0) =\frac{1}{2}
\begin{pmatrix}
 1 & 0 & -i_1 \\
0 & 0 & 0 \\
i_1 & 0 & 1
\end{pmatrix},
$$
$$c = \frac{1}{2}(0,1,1 \mid i_2,0,0) =\frac{1}{2}
\begin{pmatrix}
0 & 0 & 0 \\
0 & 1 & i_2 \\
0 & -i_2 & 1
\end{pmatrix},~
d = \frac{1}{9}(1,4,4\mid 4i_4,2i_3,2i_6) =\frac{1}{9}
\begin{pmatrix}
 1 & 2i_6 & -2i_3 \\
-2i_6 & 4 & 4i_4 \\
2i_3 & -4i_4 & 4
\end{pmatrix}.
$$

\begin{prop}\label{p:albert} The Albert algebra $A(\mathbb{F})$ is an axial $\mathbb{F}$-algebra of Jordan type $\frac{1}{2}$ generated by four primitive axes $a,b,c,d$.
\end{prop}
\begin{proof}
We claim that as a basis of $A(\mathbb{F})$ we can take the following 27 elements:

$$a, b, c, d,ab=\frac{1}{8}(2,0,0~|~i_3, i_1, i_0), ac=\frac{1}{8}(0,2,0~|~i_2, -i_6, i_0), ad=\frac{1}{36}(2,8,0~|~-2i_1+4i_4, 2i_3-4i_5,5i_0+4i_6),$$
$$bc=\frac{1}{8}(0,0,2~|~i_2,i_1,i_4), bd=\frac{1}{36}(2,0,8~|~4i_4+2i_5,5i_1+4i_3,2i_6-4i_2),
cd=\frac{1}{18}(0,4,4~|~4i_2+4i_4,i_0+i_3,i_6-i_5),
$$
$$
a(bc)=\frac{1}{32}(0,0,0~|~i_2+i_3, i_1-i_6, 2i_4),
b(ac)=\frac{1}{32}(0,0,0~|~i_2+i_3, -2i_6, i_0+i_4),
$$
$$
c(ab)=\frac{1}{32}(0,0,0~|~2i_3, i_1-i_6, i_0+i_4), a(bd)=\frac{1}{144}(4,0,0~|~5i_3-4i_1+4i_4+2i_5, 5i_1+4i_3+2i_4-4i_5, 2i_0-8i_2+4i_6)
$$
$$
a(cd)=\frac{1}{72}(0,8,0~|~-1-i_1+4i_2+4i_4, i_0+i_3-4i_5-4i_6, 4i_0-2i_5+2i_6)
$$
$$
b(ad)=\frac{1}{144}(4, 0, 0~|~-2i_1+5i_3+4i_4+4i_5, 2i_1+4i_3-8i_5, 2+5i_0-4i_2+4i_6)
$$
$$
b(cd)=\frac{1}{72}(0, 0, 8~|~4i_2+4i_4+i_5+i_6, 2i_0+4i_1+2i_3, -4i_2+4i_4-i_5+i_6)
$$
$$
c(ad)=\frac{1}{144}(0, 16, 0~|~-4i_1+8i_2+8i_4, 4i_0+2i_3-4i_5-5i_6, 5i_0-4i_3-2i_5+4i_6)
$$
$$
c(bd)=\frac{1}{144}( 0, 0, 16~|~ 8i_2+8i_4+4i_5, 4+2i_0+5i_1+4i_3, -4i_2+5i_4-4i_5+2i_6)
$$
$$
(ab)(cd)=\frac{1}{288}(0,0,0~|~-1-i_1+8i_3+i_5+i_6, 2i_0+4i_1-i_2+2i_3-i_4-4i_5-4i_6, -1+4i_0-i_1-4i_2+4i_4-2i_5+2i_6)
$$
$$
(ac)(bd)=\frac{1}{576}(0, 0, 0~|~2+4i_0-4i_1+8i_2+5i_3+8i_4+4i_5, 4+2i_0+2i_4-4i_5-10i_6, 2i_0+2i_1 -8i_2+4i_3+5i_4-4i_5+4i_6)
$$
$$
d(a(bc))=\frac{1}{576}(0, 0, 0~|~2+8i_2+8i_3+2i_5-4i_6, -8+2i_0+5i_1-2i_4-5i_6, -2-4i_2+4i_3+10i_4-2i_5)
$$
$$
d(b(ac))=\frac{1}{576}(0, 0, 0~|~4-2i_1+8i_2+8i_3-2i_6, -4+2i_0-2i_4-4i_5-10i_6, -2+5i_0+8i_3+5i_4-2i_5)
$$
$$
a(b(cd))=\frac{1}{288}(0, 0, 0~|~-2-2i_1+4i_2+4i_3+4i_4+i_5+i_6, 2i_0+4i_1+i_2+2i_3+i_4-4i_5-4i_6,
-8i_2+8i_4-2i_5+2i_6)
$$
\begin{multline*}
(ab)(c(ad))=\frac{1}{2304}(10, 10, 0~|~-4+4i_0-2i_1+5i_2+21i_3+4i_4+4i_5+2i_6, 4+8i_0+5i_1-2i_2+8i_3-4i_4-16i_5-18i_6, \\ 2+26i_0-4i_1-4i_2-8i_3+3i_4-4i_5+8i_6)
\end{multline*}
\begin{multline*}
(ab)(c(bd))=\frac{1}{2304}(10, 0, 10~|~-2-4i_0-4i_1+5i_2+21i_3+4i_4+2i_5+4i_6, 8+4i_0+26i_1-4i_2+8i_3+2i_4-4i_5-3i_6,\\ -4+5i_0-2i_1-16i_2-4i_3+18i_4-8i_5+8i_6)
\end{multline*}
\begin{multline*}
(ac)(b(cd))=\frac{1}{1152}(0, 8, 8~|~-1+4i_0-i_1+16i_2+8i_4+2i_5+2i_6, 4+i_0+4i_1+i_2+i_3+i_4-4i_5-12i_6,\\ 1+4i_0+i_1-8i_2+4i_3+12i_4-4i_5+4i_6).
\end{multline*}
All calculations are straightforward and can be done by hand\footnote{Calculations for this proof can be found in \url{https://github.com/AlexeyStaroletov/AxialAlgebras/blob/master/JordanFactors/AlbertAlgebra.g}}.
Now we write $27\times 27$ matrix of coefficients of these 27 elements with respect to the standard basis of $A(\mathbb{F})$ (i.e. $(1,0,0~|~0,0,0),\ldots, (0,0,0~|~0,0,i_6)$).
Using GAP, we find that the determinant of this matrix equals
$\frac{1}{2^{78}\cdot 3^{36}}$ and hence 27 elements form a basis of $A(\mathbb{F})$.

Since $A(\mathbb{F})$ is known to be a Jordan algebra and $a,b,c,d$ are its idempotents, Lemma~\ref{l:peirce} implies that
each of these elements gives a Peirce decomposition of the algebra.
According to~\cite[Section~4]{J60}, an idempotent $e$ in $A(\mathbb{F})$ is a primitive axis iff $Tr(e)=1$, where $Tr$ means the trace of $e$, i.e. the sum of elements on its diagonal.
Therefore, we infer that $a,b,c,d$ are primitive axes generating $A(\mathbb{F})$. This completes the proof of the proposition.
\end{proof}

\begin{cor} If the characteristic of $\mathbb{F}$ equals zero, then
  $A(\mathbb{F})$ is not a factor of any of the Matsuo algebras.
\end{cor}
\begin{proof}
 Suppose $(G,D)$ is a 3-transposition group and $M=M_\eta(G,D)$ is its Matsuo algebra for $\eta\in\mathbb{F}\setminus\{0,1\}$ such that $A(\mathbb{F})$ is a factor of $M$. Since $A(\mathbb{F})$ is simple, we can assume that $(D)$ is connected. Now $\dim_{\mathbb{F}}A(\mathbb{F})=27$ and the result follows from Proposition~\ref{p:albert} and Theorem~\ref{t:main}.
\end{proof}

\section*{Acknowledgments}
The authors would like to thank Prof. Sergey Shpectorov for
helpful comments and remarks.

\Addresses
\end{document}